\documentclass[11pt]{amsart}
\usepackage{amsfonts}

\usepackage[all,cmtip]{xy}
\usepackage{amsmath}
\usepackage{amsthm}
\usepackage{amssymb}
\usepackage{enumitem}	
\newtheorem{thm}{Theorem}[section]
\newtheorem{conj}{Conjecture}[section]
\newtheorem*{conj*}	{Conjecture}
\newtheorem{prop}[thm]{Proposition}
\newtheorem*{definition*}         {Definition}
\newtheorem{lemma}[thm]{Lemma}

\newtheorem{cor}[thm]{Corollary}

\newtheorem*{remark}{Remark}
\theoremstyle{remark}
\newcommand*{\Oo}{\mathcal{O}}
\newcommand*{\Q}{\mathbb{Q}}

\newcommand*{\Z}{\mathbb{Z}}

\newcommand*{\A}{\mathcal{A}}
\newcommand*{\R}{\mathbb{R}}

\newcommand*{\C}{\mathbb{C}}
\newcommand*{\F}{\mathbb{F}}

\newcommand*{\Disc}{\textrm{Disc}}
\newcommand*{\Trace}{\textrm{Trace}}

\newcommand*{\ra}{\rightarrow}

\newcommand*{\qng}{(q^n)^{\frac{g(g+1)}{2}}}
\newcommand*{\Cl}{\mathcal{C}\ell}

\def\Sp{{\rm Sp}}

\def\GSp{{\rm GSp}}
\def\Disc{{\rm Disc}}

\def\im{{\rm Im}}

\usepackage{amscd,amssymb,amsmath}
\usepackage{color}

\title{Unlikely intersections in finite characteristic}
\author{Ananth N. Shankar}
\author{Jacob Tsimerman}
\begin{document}
\maketitle
\begin{abstract}

We present a heuristic argument based on Honda-Tate theory against many conjectures in `unlikely intersections' over the algebraic closure of a finite field; notably, we conjecture that every 
abelian variety of dimension 4 is isogenous to a Jacobian. Using methods of additive combinatorics, we are able to give a negative answer to a related question of Chai and Oort where the 
ambient Shimura Variety is a power of the modular curve. 

\end{abstract}

\section{Introduction}

In \cite{CO}, Chai and Oort ask the following (folklore) question: \emph{ For every algebraically closed field $k$ and $g\geq 4$, does there exist an abelian variety $A$ over $k$ which is not 
isogenous to the Jacobian of a stable curve}. It is natural to interpret this question by looking the moduli space $\A_g$ of principally polarized abelian varieties of dimension $g$: let $\tau_g
\subset \A_g$ be the Torelli locus of Jacobians of Stable curves, and $T_n\tau_g$ be the locus abelian varieties with a degree $n$ isogeny to the Torelli locus. Then an equivalent 
formulation of the conjecture is the statement that there exists a $k$-point of $\A_g$ outside the countable union
of these proper subvarieties. In other words, $$\bigcup_n T_n\tau_g(k)\neq \A_g(k).$$ In this form, it becomes natural to 
replace $\tau_g$ by any proper subvariety and retain the statement of the conjecture. Using this interpretation, it becomes a simple matter to verify the conjecture
for uncountable fields $k$, and even those with sufficiently large($\geq\frac{g^2+g}{2}$) transcendence degree over their prime field. However, the cases of 
$\overline{\Q}$ and $\overline{\F}_p$ prove substantially more difficult. The characteristic 0 case was settled by the second-named author in \cite{Tsim1}, using a strategy of Chai and Oort 
outlined in \cite{CO}
relying on the Andre-Oort conjecture\footnote{Now a theorem for $\A_g$ \cite{Tsim2}} about CM points in subvarieties. Though an analogous 
strategy cannot work over $\overline{\F}_p$ as every point is a CM point, the prevailing opinion seems to be that the answer to the question should nevertheless be `yes'. 
In this paper, we present a heuristic probabilistic argument for deciding such conjectures, and conjecture based on this heuristic the answer should be `no'  over $\overline{\F}_p$
for $g=4$. 

\subsection{Plan for the rest of the paper}

In \S2, we present our heuristic, and lay out our main conjectures. In \S3 we discuss sizes of isogeny classes of Principally Polarized Abelian Varieties. We detail some conjectures about 
sizes of isogeny classes over finite fields and present some partial results. In \S4,  we provide evidence for our conjectures by providing an explicit hypersurface in $X(1)^{270}$ which we prove intersects every isogney class. Our proof relies heavily on results in additive combinatorics. 

\subsection{Acknowledgements}

It is a pleasure to thank Boris Bukh for help with the construction in section 4. It is also a pleasure to thank Vivek Shende for many helpful discussions and clarifications regarding sizes of isogeny orbits, and Hunter Spink for helpful discussions about section 3. Thanks also to Igor Shparlinski and Oliver Roche-Newton for pointing out a several references to us, leading to 
an improvement of the constant in Theorem \ref{modular}.

\section{The Heuristic}

\subsection{Honda-Tate Theory}

In this section we describe our heuristic. We first restrict to the setting of ordinary abelian varieties: Let $V\subset \A_g$ be a subvariety of dimension $d$ intersecting the ordinary Newton stratum, and let $V$ be defined over the finite field $\F_{q}$. Our idea is to count the number of ordinary $\F_{q^n}$-isogeny classes weighted by their size, that show up in 
$\A_g(\F_{q^n})$ for large $n$. By the Honda-Tate theorem, isogeny classes are in bijection with certain type of $q^n$-weil numbers. One can count
these rather easily by looking at weil-polynomials\footnote{Remember our dimension $g$ is fixed and the finite field $\F_{q^n}$ is growing; otherwise this would be quite thorny!} and arrive at an answer of roughly $q^{n\frac{g^2+g}{4}}$. As such, one would expect the size of each ordinary isogeny class to also be roughly $q^{n\frac{g^2+g}{4}}$,
to account for the $q^{n\frac{g^2+g}{2}}$ points in $\A_g(\F_{q^n})$.

Now, we have a natural map from $V(\F_{q^n})$ to isogeny classes of Abelian varieties, and the underlying assumption we make for our heuristic is 
to treat this as a random map. Now, $V(\F_{q^n})$ has roughly $q^{nd}$ points, which leads us to a natural conjecture. 

There is a caveat to be had: if $V$ is contained in the mod-$p$ reduction of a characteristic-0 proper Shimura subvariety of $\A_g$, then there will be Weyl CM points that neither $V$ nor its reduction can contain. We rule this out by insisting that for a prime $\ell$ not dividing  $q$, the $\ell$-adic monodromy of $V$ is  Zariski-dense in the symplectic group $\GSp_{2g}$. Thus, our conjecture is:

\begin{conj}[The ordinary stratum]\label{mainconjord}

 Let $V\subset\A_g$ be an irreducible subvariety of dimension $d$, whose $\ell$-adic monodromy is Zariski dense in $\GSp_{2g}$. If $d\geq\frac{g^2+g}{4}=\frac{\dim \A_g}{2}$, then for each ordinary abelian variety $A$ over $\overline{\F}_p$, there exist infinitely many points in $V(\overline{\F}_p)$ corresponding to abelian varieties which are isogenous to $A$. If, on the other hand,  $V$ does not satisfy either the monodromy condition or the dimension condition, then there exists an abelian varietiy $A$ over $\overline{\F}_p$ such that no points in $V(\overline{\F}_p)$ corresponding to abelian varieties which are isogenous to $A$.
 \end{conj}
 

\subsection{Non-ordinary Newton strata}
We now deal with the other Newton strata. To that end, let $A$ be a principally polarized abelian variety, let $W$ be its Newton polygon, and let $N(W)$ be the open Newton stratum of $\A_g$ consisting of all abelian varieties whose Newton polygon is $W$. By work of Oort (see \cite[\S 5]{Oort}), every Newton stratum 
 decomposes into an almost product of what he terms as ``Central leaves'' and ``Isogeny leaves''. 
 
 The central leaf through $A$ consists of all abelian varieties in $N(W)$ whose $p$-divisible group is isomorphic to $A[p^{\infty}]$. The isogeny leaf through $A$ is a maximal irreducible subscheme of $\A_g$, conisting of all abelian varieties $A'$ in $N(W)$ with the property that $A'$ is isogenous to $A$ through an isogeny with the kernel being an extension of $\alpha_p$ group schemes. The dimensions of the central leaves and isogeny leaves through $A$ depend solely on the Newton polygon of $A$. Using the computations in \cite[\S 5]{Oort1}, we see that the number of Weil $q^n$-numbers with Newton polygon $W$ is  of size $O(q^{n\frac c 2})$, where $c$ is the dimension of the central leaf associated to $A$ (it is exactly of size $q^n{\frac c 2}$ when $n$ is sufficiently divisible).

Suppose that $V \subset \A_g$ is a $d$-dimensional subvariety that intersects $N(W)$. Then the assumption we make for our heuristic is to treat the map from $V(\F_{q^n})$ to the set of Weil $q^n$-numbers with Newton polygon $W$ as a random map, but we consider the dimension of  the projection of $V \cap N(W)$ onto a central leaf instead of the dimension of $V \cap N(W)$. We impose this condition because we wish to rule out the case of $V \cap N(W)$ being a positive dimensional fibration over a subvariety over the central leaf. Analogous to Conjecture \ref{mainconjord}, our expectation is: 

\begin{conj}[Arbitrary Newton strata]\label{mainconj}
Suppose that $W \subset \A_g$ is a closed subvariety as above, and let $d$ be the dimension of the projection of $\dim V \cap N(W)$ onto a central leaf.  If $d>c/2$, and the $\ell$-adic monodromy of $V\cap N(W)$ is Zariski dense in $\GSp_{2g}$, then for each abelian variety $A$ over $\overline{\F}_p$ with Newton polygon $W$, there exist infinitely many points in $V(\overline{\F}_p)$ 
 corresponding to abelian varieties which are isogenous to $A$.  If, on the other hand,  $V$ does not satisfy either the monodromy condition or the dimension condition, then there exists an abelian varietiy $A$ over $\overline{\F}_p$ with Newton polygon $W$ such that no points in $V(\overline{\F}_p)$ corresponding to abelian varieties which are isogenous to $A$.

 \end{conj}
 
\subsection{Abelian varieties isogenous to Jacobians}
We expect that the Torrelli locus $\tau_g$ satisfies the conditions required by Conjecture \ref{mainconjord} if $g \leq 9$ (this is certainly the case for the ordinary locus, and the Newton stratum having codimension 1, where the central leaf equals the whole Newton stratum). Therefore, we expect that every ordinary abelian variety over $\overline{\F}_p$ of dimension $\leq 9$ is isogenous to the Jacobian of some curve. In the case where $g = 4$, the Newton strata which have codimension $\ge 2$ are easy to handle, as we demonstrate below: 
\begin{prop}
Let $A$ be a 4-dimensional principally polarised abelian variety over $\overline{\F}_p$ contained in a Newton stratum of codimension $\geq 2$. Then $A$ is isogenous to a jacobian. 
\end{prop}
\begin{proof}
Let $x \in \A_4(\F_q)$ correspond to the point $A$. We claim that $\tau_4 \subset \A_4$ passes through a point isogenous to $x$. 

First, note that $\tau_4$ is an effective divisor of $\A_4$, whose picard group is $\Z$. Therefore, the class of $\tau_4$  is ample. This implies that $\tau_4$ will intersect the closure of every positive dimensional subvariety of $\A_4$ in its Baily-Borel-Satake compactification, which we call $\overline{\A_4}$. 

By \cite[Proposition 4.11]{Oort}, $I_x$ is proper (Oort proves this for every Newton stratum), and is positive dimensional (whence the restriction that the Newton stratum of $A$ has codimension $\geq 2$). Because $\mathcal{I}_x$ is positive dimensional, its closure has to intersect the closure of $\tau_4$ in $\overline{\A_4}$. However, because $\mathcal{I}_x$ is proper, it is closed in $\overline{\A_4}$. Therefore, this intersection must happen in the interior. The proposition follows. 
\end{proof}

We thus end the section with the following conjecture: 

\begin{conj}\label{Abjacob}
Every 4-dimensional abelian variety over $\overline{\F}_p$ is isogenous to the Jacobian of some curve. 
\end{conj}

\begin{remark} 

While our conjectures imply that for $g\leq 9$ every ordinary abelian variety should be isogenous to a Jacobian, there are other Newton polygon strata which suggest the opposite unless they have excessive intersection with the Torelli locus. It seems that it is hard to compute the precise dimension of these intersections, so we are hard pressed to formulate a general conjecture in this case.
\end{remark}

\section{Isogeny classes of Abelian varieties}
Let $A$ be principally polarised abelian variety of dimension $g$ defined over $\F_q$, which is ordinary and geometrically simple. Let $\alpha$ denote the corresponding Weil-number. Define $K =\Q(\alpha)$. The field $K$ is a CM field, and define $K^+$ to be the maximal totally real subfield of $K$ (of degree $g$ over $\Q$).  Let $\alpha_1, \alpha_2, ..., \alpha_g, q/\alpha_1, \hdots , q/\alpha_g$ denote the image of $\alpha$ under the $2g$ complex embeddings of $K$. Let $\theta_1, \hdots \theta_g$ denote the arguments of $\alpha_1, \hdots \alpha_g$. 

Define $I(q^n,A)$ denote the set of principally polarised abelian varieties defined over $\F_{q^n}$ that are isogenous to $A$, and let $N(q^n,A) = \# I(q^n,A)$. 
\begin{conj} \label{conjsize}
We have $N(q^n,A) \le (q^{n(g(g+1))/2})^{\frac{1}{2}+ o(1)}$. Moreover, for a positive density of integers $n$, we have 
$N(q^n,A) = q^{n(g(g+1))/2})^{\frac{1}{2}+ o(1)}$. 
\end{conj}

In this section, we prove Conjecture \ref{conjsize} in the case of elliptic curves, prove the lower bounds for arbitrary $g$ in the case where there exist $g$ conjugates of the Weil-number of $A$ which are multiplicatively independent. Zarhin in \cite{Z} proves that this always happens if $g = 2$ or $3$. 

\subsection{Deligne's category}
We recall Deligne's description of the category of ordinary abelian varieties over $\F_q$. To ease the exposition, we work with abelian varieties which are geometrically simple. Let $R_n = \Z[\alpha^n,(q/\alpha)^n]$, which is an order inside the CM field $K$. We say that $I \subset K$ is a fractional ideal of $R_n$, if $I$ is a non-trivial finitely generated $R$-submodule of $K$. Let $\mathcal{I}_n$ denote the set of fractional ideals of $R_n$. We define $\mathcal{C}_n = \displaystyle{\mathcal{I}_n/\sim}$, where $I_1 \sim I_2$ if there exists some $x \in K^{\times}$ such that $xI_1 = I_2$. The following result is proved in \cite{Del}

\begin{thm}[Deligne]
The set of abelian varieties defined over $\F_{q^n}$ isogenous to $A$ is in bijection with the elements of $\mathcal{C}_n$.
\end{thm}

\subsubsection{Polarizations} 
In his paper \cite{Howe}, Howe describes polarizations on ordinary abelian varieties in terms of their Deligne modules. We recall this description for later use. Let $\Phi$ denote the CM-type on $K$ induced by $A$. Let $I \in \mathcal{C}_n$. By \cite[Proposition 4.9]{Howe}, a polarization on the abelian variety corresponding to $I$ is given by a $\Z$-valued bilinear form on $I$ which is of the form 
$$(x,y) \mapsto \Trace_{K/\Q}(\lambda x \overline{y}),$$
where:
\begin{enumerate}
\item $\overline{y}$ corresponds to complex conjugation on $K$ applied on $y$. 
\item (\ ,\ ) restricted to $I \subset K$ is integral. 
\item The element $\lambda \in K$, is purely imaginary, and has the property that $\phi(\lambda)/i$ is positive for $\phi \in \Phi$. 
\end{enumerate}
The pair $I,\lambda$ is isomorphic to the pair $\nu I, \nu \overline{\nu} \lambda$, for $\nu \in K^{\times}$. The polarization is principal if $I$ is self dual for the form. 

\subsubsection{Isogeny classes of Elliptic cures}
We now prove Conjecture \ref{conjsize} in the case $g = 1$. To that end, let $E$ denote an ordinary elliptic curve over $\F_q$. We let $\alpha, K$ and $R$ denote the same objects as above. 
\begin{thm}\label{isogorbits}

$N(q^n,E)\leq (q^n)^{\frac12 +o(1)}$. Moreover, for a density-one set of $n$, we have\footnote{The upper bound in a more precise form is due to Lenstra\cite[Prop 1.19]{Len}} $N(q^n,E)=(q^n)^{\frac12+o(1)}$. 

\end{thm}

\begin{proof}

Write $\Oo_K = \Z[\beta]$ where $\beta$ is either $\sqrt{D}$ if $D\neq0\mod{4}$ or $\frac{1+\sqrt{D}}{2}$ if $D=1\mod{4}$. Then the index of $R_n$
in $\Oo_K$ is $i_n=c\cdot \im(\alpha^n)$ where $c$ is either $D^{-\frac12}$ or $2D^{-\frac12}$. Alternatively, writing $\alpha=q^{\frac12}e(\theta)$ we 
see that $\im(\alpha^n)=q^{\frac{n}2}\sin(n\theta)$. Note that since $E$ is an ordinary elliptic curve, $\theta/\pi$ is irrational. 

Every order $\Oo \subset K$ is Gorenstein, and hence every element of $\mathcal{C}_n$ lies in $\Cl(\Oo)$ for some unique $\Oo$. Therefore, we have 
$$\mathcal{C}_n = \displaystyle{\bigcup_{R_n \subset \Oo} \Cl(\Oo)}.$$ 

Now, for the order $\Oo_d$ of index $d$ in $\Oo_K$, the class number $\Cl(\Oo_m)$ satisfies (see \cite[Ex. 4.12]{S}) 
$$h(\Oo_d) = d[\Oo_K^{\times}:\Oo_d^{\times}]^{-1}h(\Oo_K)\cdot \prod_{p\mid d} (1-\left(\frac{D}{p}\right)).$$ 
Thus, we see that $$N(q^n,E) = \sum_{d\mid i_n}h(\Oo_d)\leq i_n^{o(1)}h(R_n) = i_n^{1+o(1)}\leq (q^n)^{\frac12+o(1)}$$ as desired. 

On the other hand, for a set density one of integers $n$ we have that $sin(n\theta)>1/n$, and then 

$$N(q^n,E)\geq h(R_n)=i_n^{1+o(1)} \geq (q^n/n)^{\frac12+o(1)} = (q^n)^{\frac12 + o(1)}.$$

\end{proof}


\subsection{Class groups and isogeny classes}

We retain notation from earlier in the section. Let $R^+_n = \Z[\alpha^n + q/\alpha^n]$, an order of the totally real field $K^+$. For brevity, let $R^+$ denote $R_1^+$. 

\begin{prop}\label{isogclassgroup}
The subset of $I(q,A)$ with endomorphism ring exactly equal to $R$ is either empty, or in bijection with the kernel of the norm map 
$$N: \Cl(R) \rightarrow \Cl^+(R^+).$$
Here, $\Cl^+(R^+)$ is the narrow class group of the totally real order $R^+$. 
\end{prop}
\begin{proof}
The set of abelian varieties defined over $\F_q$ with endomorphism ring $R$ is in bijection with $\Cl(R)$. Suppose that an invertible ideal $I$ is self dual for some $(\ ,\ )$, i.e. there exists a principally polarised abelian variety with endomorphism ring $R$. For any invertible $R$-ideal $J$, the lattice dual to $IJ$ contains $I\overline{J}^{-1}$. Because $R$ is gorenstein and all the ideals in question are invertible, we have $[I:IJ] = [R:J] =[\overline{J}^{-1}:R] = [I\overline{J}^{-1}:I]$. Therefore, the lattice dual to $IJ$ equals $I\overline{J}^{-1}$. 

The abelian variety corresponding to $IJ$ will be principally polarized if and only if there exists a totally positive $\beta \in R^+ \otimes \Q$ such that $IJ$ is self-dual for the form 
$$(x,y) \mapsto \Trace_{K/\Q}(\beta\lambda x \overline{y}),$$
that is if $IJ = I\overline{J}^{-1}/\beta$, that is $J$ is in the kernel of the norm map. 
\end{proof}

\begin{prop}\label{algind}
Suppose that $q^{1/2}, \alpha_1, \hdots, \alpha_g$ are multiplicatively independent. Then, for a positive density of $n$, there exists a principally polarized abelian variety in $I(q^n,A)$ having endomorphism ring $R_n$. 
\end{prop}
\begin{remark}
In fact, there exists $c \in \Cl^+(R^+)$ with the following properties: 
\begin{enumerate}
\item The class of $c$ in $\Cl(R^+)$ is trivial. 
\item The set of principally polarized abelian varieties isogenous to $A$ with endomorphism ring exactly $R$ is in bijection with the pre-image of $c$ in $\Cl(R)$. 
\end{enumerate}
The content of Proposition \ref{algind} is to prove that under the above conditions on $\alpha$, there exists a set of integers $n$ with positive density such that $c \in \Cl^+(R^+)$ is the trivial element. 
\end{remark}

\begin{proof}[Proof of Proposition \ref{algind}]
Recall that the abelian variety $A$ gives rise to a CM type $\Phi$ of $K$. We retain the notation of Proposition \ref{isogclassgroup}. Without loss of generality, let $\phi_j(\alpha) = \alpha_j$, where $\phi_j \in \Phi$ for $1 \le j \le g$. Let $g_n(x)$ denote the minimal polynomial of $\alpha^n + (q/\alpha)^n$ over $\Q$. The trace dual of $R_n$ is $[(\alpha^n - (q/\alpha)^n)g_n'(\alpha^n + (q/\alpha)^n)]^{-1}$. 

We claim that the set of $n$ for which the bilinear form on $R_n$ given by 
$$(x,y) \mapsto \Trace_{K/\Q}(\lambda_n x \overline{y}),$$
with $\lambda = (\alpha^n - (q/\alpha)^n)g_n'(\alpha^n + (q/\alpha)^n)$ satisfies the polarization conditions has a density of  $1/2^g$. Indeed, $\lambda_n$ is purely imaginary, and we only need verify when the positivity conditions hold. 

Recall that we defined $\theta_j$ to be the argument of $\alpha_j = \phi_j(\alpha)$. Therefore, $\phi_j(\lambda_n)/i = (4q)^{g/2}\sin(n\theta_j)\displaystyle{\prod_{k \neq j}(\cos(n\theta_j) - \cos(n\theta_k))}$. Thus, the result follows if we prove that the set of positive integers  which for all $1 \le j \le g$ satisfy 
$$\displaystyle{\sin(n\theta_j)\prod_{j \neq k}(\cos(n\theta_j) - \cos(n\theta_k))} > 0$$
has a density of $1/2^g$. We have assumed that the set $\{1, \alpha_1, \hdots \alpha_g \}$ is a multiplicatively independent set. It follows that the elements $ 2 \pi, \theta_1, \hdots \theta_g $ are $\Q$-linearly independent, and thus the sequence $( n\theta_1, \hdots n\theta_g) $ is equidistributed modulo the box $[0,2\pi]^g$. The locus of points in $[0,2\pi]^g$ which satisfy the necessary inequalities is open, and has measure $1/2^g$. The proposition follows. 
\end{proof}

\begin{remark}
We consider the set of points $(\theta_1 \hdots \theta_g) \in \R^g$ which have the property that for some $1 \le j \le g$, 
$$\displaystyle{\sin(n\theta_j)\prod_{j \neq k}(\cos(n\theta_j) - \cos(n\theta_k))} = 0.$$
This set is clearly a union of hyperplanes of the form $\theta_j = m\pi$ and $\theta_j \pm \theta_k = 2m\pi$, as $m$ varies over $\Z$. In fact, the space $\R^g$ along with this set of hyperplanes is exactly the affine apartment of the simple group $\Sp_{2g}$. 
\end{remark}

We have identified a subset of $I(A,q^n)$ with a fiber of the map $N: \Cl(R_n) \rightarrow \Cl^+(R^+_n)$, and have proved that if $1, \alpha_1, \hdots \alpha_g$ are multiplicatively independent, then for a positive proportion of $n$, the fiber is non-empty. The following theorem follows from these facts, and from Lemma \ref{bigclassgroup} below: 

\begin{thm}
Suppose that $q^{1/2}, \alpha_1, \hdots \alpha_g$ are multiplicatively independent. Then $N(A,q^n) \ge (q^{n(g(g+1))/2})^{\frac{1}{2}+ o(1)}$.
\end{thm}

\begin{lemma}\label{bigclassgroup}
For a density-one set of positive integers $n$, we have $\frac{\#\Cl(R_n)}{\#\Cl^+(R^+_n)} = (q^n)^{\frac{g(g+1)}{4} + o(1)}$. 
\end{lemma}
\begin{proof}
It suffices to prove the result with $\Cl(R_n^+)$ in place of $\Cl^+(R_n^+)$. For ease of notation, let $\alpha_{g+j} \in \C$ denote $q/\alpha_j$, for $1\le j\le g$. We also define $R'_n = \Z[\alpha^n]$. 

As $n$ approaches $\infty$, the class group of $R_n$ (and of $R_n^+$) is well approximated by the square-root of its discriminant, so it suffices to prove that for a positive density of integers $n$, $\Disc(R_n)/\Disc(R^+_n)$ has the same order of magnitude as $\qng$. 

We have that $\Disc(R_n^+) = \displaystyle{\prod_{j<k \le g} (\alpha_j^n + \alpha_{g+j}^n - \alpha_k^n - \alpha_{g+k}^n)^2}$. The discriminant of $R'_n$ equals 
$$\Disc(R'_n) = \displaystyle{\prod_{j<k \le 2g} (\alpha_j^n  - \alpha_k^n)^2}.$$ 
Translating this in terms of polar coordinates, we see that $$\Disc(R^+_n) = \displaystyle{\prod_{j<k \le g} 4q^n(\cos n\theta_j - \cos n\theta_k)^2},$$
 and that 
 $$\Disc(R'_n) = \displaystyle{\prod_{j<k \le 2g} q^n(\cos n\theta_j + i\sin n\theta_j- \cos n\theta_k - i\sin n\theta_k)^2}.$$ 
By Lemma \ref{d} below, we have that $\Disc(R') = q^{ng(g-1)}\Disc(R)$. Therefore, the quotient of the orders of the classgroups is approximated by
$$\frac{\#\Cl(R_n)}{\#\Cl(R_n^+)} =  \displaystyle{\frac{q^{n2g(2g-1)/4}\prod_{j<k \le 2g} (\cos n\theta_j + i\sin n\theta_j- \cos n\theta_k - i\sin n\theta_k)}{q^{ng(g-1)/2}q^{ng(g-1)/4}\prod_{j<k \le g}(\cos n\theta_j - \cos n\theta_k)}}.$$

It is easy to see that the term involving sines and cosines is absolutely bounded above, and is greater than $1/n$ for a density-one set of positive integers. The result follows. 
\end{proof}

\begin{lemma}\label{d}
We have that $\Disc(R'_n) = q^{ng(g-1)}\Disc(R_n)$.
\end{lemma}
\begin{proof}
We do the computation for $n=1$, as there is no loss in generality. Further, because $\alpha$ is a local unit at every prime $l$ such that $l \neq p$, the discriminants can differ only at $p$. Suppose that $R = R_1$ and $R^+ = R^+_1$.

Let $K^+ \otimes \Q_p = K^+_p$. Because $\alpha$ is ordinary, we have that $K \otimes \Q_p = K^+_p \oplus K^+_p$. The image of $\alpha$ in $K \otimes \Q_p$ will be of the form $(\gamma,q/\gamma)$, where $\gamma$ is a unit in $\Oo_{K^+_p}$. The image of $\beta$ is $(q/\gamma,\gamma).$

The order $R'$ at $p$, $R' \otimes \Z_p$ splits into a direct sum $\Oo_1 \oplus \Oo_2$, corresponding to the decomposition $K \otimes \Q_p = K^+_p \oplus K^+_p$. This is because $R' = \Z_p[\alpha]$, and $\alpha^n$ converges $p$-adically to $(1,0) \in K^+_p \oplus K^+_p$. The same is true of $R \otimes \Z_p$. Therefore, $R' \otimes \Z_p = \Z_p[\gamma]\oplus \Z_p[q/\gamma]$, and $R \otimes \Z_p = \Z_p[\gamma] \oplus \Z_p[1/\gamma]$. The latter statement is true because $\gamma$ is a unit in $\Z_p[\gamma]$. We observe that $\Disc(\Z_p[q/\gamma]) = q^{g(g-1)}\Disc(\Z_p[1/\gamma])$ which finishes the proof in this case.

\end{proof}




\section{ Proofs for powers of the modular curve }

The analysis in the previous section works just as well (in fact its quite a bit easier) for powers of the modular curve $X(1)^n$ parametrizing $n$-tuples
of elliptic curves. In this case, however, there is also a large advantage in that the isogeny orbits break up into product sets which allow us to use methods
of additive combinatorics. This section is devoted to the proof of the following theorem: 

\begin{thm}\label{modular}

For any $n\geq 270$, there exists a proper subvariety $V\subset X(1)^{n}_{\overline{\F}_p}$ such that for every point $(j(E_1),\dots,j(E_n))$ there exist elliptic curves $(E'_1,\dots,E'_{n})$ such that $(j(E'_1),\dots,j(E'_n))\in V(\overline{\F}_p)$, where $E_i$ isogenous to $E'_i$ for $1\leq i\leq n$.

\end{thm}

Chai and Oort (see \cite[\S 4]{CO}) ask whether there exists a proper subvariety of product of modular curves which intersects every ordinary isogeny class. The content of Theorem \ref{modular} is that the answer to the above question is {\it yes}. 

The idea of the proof is to first use sum-product theorems to grow our product-sets, and then use billinear forms and a standard trick with completing our sum over a finite field.  
We begin by recalling some lemmas, starting with Rusza's triangle inequalities:
\begin{lemma}

For every abelian group $G$ and triple of sets $A,B,C\subset G$ we have $$|A\pm C||B|\leq |A\pm B||B\pm C|$$ where the theorem holds for all 8 possible
choices of signs.

\end{lemma}

\begin{proof}

\cite[Theorem 1.8.1 and 1.8.7]{Ruz}

\end{proof}

Plugging in $A=C$ in the above gives the following:

\begin{cor}\label{rusza}

For every abelian group $G$ and pair of sets $A,B\subset G$ $$|A\pm A||B|\leq |A\pm B|^2$$ for all 4 possible choices of signs.

\end{cor}

We shall also require sum-product results for finite fields not of prime order:

\begin{lemma}\label{sumprod}

Let $F$ be a finite field. Suppose that $A\subset F$ is such that for any field $F_1\subset F$ and constant $c\in F$ we have
$$|A\cap (cF_1)|<\max(|A|^{\frac{9}{11}-o(1)},|F_1|^{\frac12}).$$

Then $\max(|A\cdot A|, |A+A|)\gg |A|^{\frac{12}{11}-o(1)}.$

\end{lemma}

\begin{proof} 

This is theorem 1.4 of \cite{LO}\footnote{The paper has the slightly more restrictive condition $|A\cap (cF_1)|<|F_1|^{\frac12}$, but this is amended in the thesis of Roche-Newton}, which builds on work  Katz-Shen\cite{KS}. In fact, they prove something slightly stronger by replacing the $|A|^{-o(1)}$ factor by a power of $\log|A|$. 

\end{proof}

We shall need the following slight variation of the above:

\begin{lemma}\label{sumprod1}

Let $F$ be a finite field of order $q$. Suppose that $A\subset F$ is such that $|A|=q^{\frac12 -o(1)}$, and that if there exists a subfield $F_1\subset F$ with
$[F:F_1]=2$, then $|A\cap F_1|< |A|^{\frac23}$. Then

$$\max(|A\cdot A|, |A+A|,|(A+1)\cdot (A+1)|)\gg |A|^{\frac{12}{11}-o(1)}.$$ 

\end{lemma}

\begin{proof}

By Lemma \ref{sumprod} we are done unless there exists a field $F_1\subset F$ with $[F:F_1]=2$ and $c\in F$ such that
 $$|A\cap (cF_1)|\gg |A|^{\frac{12}{11}},$$ so henceforth we assume this is the case. Now, let $A'=(A\cap cF_1) +1$. By assumption $c\not\in F_1$. 
 
Write $a_1=cb_1+1, a_2=cb_2+1$ where $b_i\in F_1$. Then 
$$(a_1a_2-1)/c^2 = b_1b_2 +\frac{1}{c}\cdot (b_1+b_2).$$ Since $\frac{1}{c}\not\in F_1$ we can recover both $b_1b_2$ and $b_1+b_2$, and thus the set
$\{b_1,b_2\}$.  It follows that if $a_1a_2=a'_1a'_2$ where $a_1,a_2,a'_1,a'_2 \in A'$ then the sets $\{a_1,a_2\}$ and $\{a'_1,a'_2\}$ are the same. Thus
$|A'\cdot A'| = |A'|(|A'|-1)/2$, and the result follows.

\end{proof}

Packaging the above into a single polynomial, we obtain the following:

\begin{lemma}\label{triples}

Let $F$ be a finite field of order $q$.  Suppose that $A,B,C\subset F$ are subsets such that $|A|,|B|,|C|$ are of size $q^{\frac12-o(1)}$, and that if there exists a subfield $F_1\subset F$ with $[F:F_1]=2$, then $\max{|A\cap F_1|,|B\cap F_1|, |C\cap F_1|}< q^{\frac13}$. Then setting $P(x,y,z)=xy+z$ , we have that
$$\max(|P(A,B,C)|, |P(A+1,B+1,C+1)|) \geq q^{\frac{23}{44}-o(1)}.$$

\end{lemma}

\begin{proof}

Note that $A,B,C$ satisfy the condition of Lemma \ref{sumprod1}. Suppose first that $|A+A|\gg q^{\frac{6}{11}-o(1)}$. Then for any non-zero $b\in B$, we have
 by Corollary \ref{rusza} that $|Ab+C|\geq q^{\frac{23}{44}-o(1)}$. Since $AB+C$ is a superset of $Ab+C$, the result follows in this case. We likewise handle
 the cases when $AA$ is large, by applying Corollary \ref{rusza} to the group $F^{\times}$. What remains is the case when $(A+1)(A+1)\geq q^{\frac{23}{44}-o(1)}.$ In this case, arguing as above gives that
 $|(A+1)\cdot(B+1)|\geq q^{\frac{23}{44}-o(1)}$, and the result follows. 
 
 \end{proof}
 
%
%
%
%
%
%
%
%
%
%
%
%
%
%
%
%
%
%

For $x,y\in \F_q^n$ we define $x\cdot y=\sum_i x_iy_i$. We shall need the following combinatorial lemma, due to Shparlinski\cite[Lemma 5]{Sh}, for which we give a  proof for
completeness.

\begin{lemma}\label{dotprod}

Let $\F_q$ be a finite field, $n$ be a positive integer, an $A,B$ be subsets of $\F_q^n$ not containing $\vec{0}$. 
Suppose that $A\times B$ does not intersect the variety $x\cdot y=0$. Then $|A|\cdot|B|\leq q^{n+2}$.
\end{lemma}

\begin{proof}

Let $f(x,y)$ be 1 iff $x\cdot y= 0$, and 0 otherwise. Then applying Cauchy-Schwartz, we get:

\begin{align*}
\frac{|A|^2|B|^2}{q^2}&=\left(\sum_{x\in A}\sum_{y\in B} \frac1q - f(x,y)\right)^2\\
&\leq |A|\sum_{x\in A}\left(\sum_{y\in B} \frac1q - f(x,y)\right)^2\\
&\leq |A|\sum_{x\in \F_q^n}\left(\sum_{y\in B} \frac1q - f(x,y)\right)^2\\
&=|A|\left(|B|^2q^{n-2}-|B|\cdot\frac2q\sum_{y\in B}\sum_{x\in\F_q^n}f(x,y) + \sum_{y,z\in B}\sum_{x\in\F_q^n}f(x,y)f(x,z)\right)\\
\end{align*}

Now, for any non-zero $y$ there are $q^{n-1}$ values of $x$ for which $x\cdot y=0$, and unless $y$ and $z$ are parallel there are $q^{n-2}$ values of $x$
for which $x\cdot y=x\cdot z = 0$. Thus,  plugging this into the above we get 

$$\frac{|A|^2|B|^2}{q^2}=|A|(q^{n-1}-q^{n-2})\#\{y,z\in B, y\mid\mid z\}\leq |A||B|q^n $$ and the result follows.

\end{proof}

Now fix $N>44$, and consider the polynomial 

$$Q(x_1,\dots,x_{6N}): = \sum_{i=1}^{N} P(x_{6i+1},x_{6i+2},x_{6i+3})P(x_{6i+4},x_{6i+5},x_{6i+6}).$$

For a vector $\vec{c}\in\{0,1\}^{2N}$, define $$Q_{\vec{c}}(x_1,\dots,x_{6N}): = \sum_{i=1}^{N} P(x_{6i+1}+c_{2i-1},x_{6i+2}+c_{2i-1},x_{6i+3}+c_{2i-1})P(x_{6i+4}+c_{2i},x_{6i+5}+c_{2i},x_{6i+6}+c_{2i}).$$

And let  $R(x_1,\dots,x_{6N})=\prod_{\vec{c}\in\{0,1\}^{2N}} Q_c.$ Combining Lemmas \ref{dotprod} and \ref{triples} we obtain:

\begin{lemma}\label{finalac}

Let $F$ be a finite field of order $q$.  Suppose that $A_1,\dots,A_{6N} \subset F$ are subsets such that $|A_i|$ are of size $q^{\frac12-o(1)}$, and that if there exists a subfield $F_1\subset F$ with $[F:F_1]=2$, then $\max_i{|A_i\cap F_1|}< q^{\frac13}$. Then $\prod_i A_i$ intersects the variety $R=0$. 

\end{lemma}

\begin{proof}

By Lemma \ref{triples} there is some $\vec{c}$ such that $$|P(A_{3i+1}+c_{i},A_{3i+2}+c_{i},A_{6i+3}+c_{i})|\geq q^{\frac{23}{44}-o(1)}$$ for all $i\leq N$.
By Lemma \ref{dotprod}, we're done if $q^{\frac{23N}{22}+o(1)} > q^{N+2}$, which holds since $N>44$.  

\end{proof}

\begin{proof} \emph{of Theorem \ref{modular}}. Note that it is enough to handle the case of $n=270$, since for $m>n$ we can just pull back under the co-ordinate projection $X(1)^m\ra X(1)^n$. Hence, we assume that $n=270$. Assume first that $\vec{x}=(x_1,\dots,x_n)\in X(1)^n$ with all the $x_i$ corresponding to ordinary elliptic curves. Then by Theorem
\ref{isogorbits}, we see that for $q$ a large power of $p$, the set $I_i=I(x_i)(\F_q)$ of points isogenous to $x_i$ over $\F_q$ is of size $q^{\frac12 -o(1)}$. Hence, by Lemma
\ref{finalac} there is a point on $\prod_i I_i$ which lies on $R=0$ (where we take $N=45$).

We thus define our variety $V$ to be the union of $R=0$ and the hypersurface $S$ defined by any of the co-ordinates being supersingular. This completes the proof.

\end{proof}


\begin{thebibliography}{10}

\bibitem{CO} Ching-Li Chai and Frans Oort, Abelian varieties isogenous to a Jacobian. {\it Ann. of Math. (2)} 176 (2012), no. 1, 589--635. 

\bibitem{Del} Pierre Deligne, Varietes abeliennes ordinaires sur un corps fini. {\it Invent. Math.} {\bf 8} 1969 238--243

\bibitem{HIS} Derrick Hart, Alex Ioevich and Joszef Solymosi, Sum-product estimates in finite fields via Kloosterman sums
{\it International Mathematics Research Notices}, Vol. 2007, , Article ID rnm007

\bibitem{Howe} Everett Howe, Principlly polarized ordinary abelian varieties over finite fields. {\it Trans. Amer. Math. Soc.} {\bf 347} (1995), no. 7, 2361--2401. 

\bibitem{KS} Net Hawk Katz and Chun-Yen Shen, Garaev's inequality in Finite fields not of prime order,
http://www.math.rochester.edu/ojac/vol3/Katz\_2008.pdf

\bibitem{Len} H.W. Lenstra, Jr., Factoring integers with elliptic curves, 
{\it Annals of Mathematics}, 126, 1987, pages 649-673.

\bibitem{LO} L.Li, O.Roche-Newton, An improved sum-product estimate for general finite fields,	
{\it SIAM J. Discrete Math}, Vol. 25, No. 3, 2011, pages 1285-1296.

\bibitem{Oort} Frans Oort, Foliations in moduli spaces of abelian varieties. {\it J. Amer. Soc.} {\bf 17} (2004), no. 2, 267 -- 296.

\bibitem{Oort1} Frans Oort, Foliations in moduli spaces of abelian varieties and dimensions of leaves. {\it Algebra, arithmetic, and geometry: in honor of Yu. I. Manin. Vol. II,} 465--501, Prog. Math., {\bf 270}, {\it Birkhauser Boston, Inc., Boston, MA 2009}

\bibitem{Ruz}
 Imre Z. Ruzsa. Sumsets and structure. In Combinatorial number theory and
additive group theory, Adv. Courses Math. CRM Barcelona, pages 87-210.
Birkh\"auser Verlag, Basel, 2009.

\bibitem{S}
Goro Shimura, Introduction to the theory of automorphic functions,
{\it Publications of the Mathematica Society of Japan 11}

\bibitem{Sh}
I.E. Shparlinski, On the additive energy of the distance set in finite fields, 
{\it Finite Fields and Their Applications}, 42, 2016, pages187?199

\bibitem{Tao}
T.Tao, expanding polynomials over finite fields of large characteristic, and a regularity lemma for definable sets,
http://arxiv.org/pdf/1211.2894v4.pdf

\bibitem{Tsim1}
Jacob Tsimerman, The existence of an abelian varietiy over $\overline{\Q}$ isogenous to no Jacobian. {\it Ann. of Math. (2)} 176 (2012), no. 1, 637--650. 

\bibitem{Tsim2}
Jacob Tsimerman, A proof of the Andre-Oort conjecture for $\A_g$. 
https://arxiv.org/pdf/1506.01466v5.pdf

\bibitem{Z}
Y.Zarhin, Eigenvalues of Frobenius endomorphisms of abelian varieties of low dimension,
{\it Journal of Pure and Applied Algebra}, 219, 2015, pages 2076-2098

\end{thebibliography}
\end{document}